\documentclass[a4paper,12pt]{amsart}

\usepackage{amssymb}
\usepackage[francais]{babel}
\usepackage{mhequ}
 \usepackage[utf8]{inputenc} 
\renewcommand{\div}{\mbox{div}}

\usepackage{cite}

\DeclareFontFamily{OT1}{rsfs}{}
\DeclareFontShape{OT1}{rsfs}{m}{n}{ <-7> rsfs5 <7-10> rsfs7 <10-> rsfs10}{}
\DeclareMathAlphabet{\mathscr}{OT1}{rsfs}{m}{n}

\newcommand{\eq}[1]{\eqref{#1}}

\newcommand{\bel}[1]{\begin{equation}\label{#1}}
\newcommand{\beal}[1]{\begin{eqnarray}\label{#1}}
\newcommand{\beadl}[1]{\begin{deqarr}\label{#1}}
\newcommand{\eeadl}[1]{\arrlabel{#1}\end{deqarr}}
\newcommand{\eeal}[1]{\label{#1}\end{eqnarray}}
\newcommand{\eead}[1]{\end{deqarr}}
\newcommand{\eea}{\end{eqnarray}}
\newcommand{\eeaa}{\end{eqnarray*}}

\newcommand{\be}{\begin{equation}}
\newcommand{\ee}{\end{equation}}

\DeclareFontFamily{OT1}{rsfs}{}
\DeclareFontShape{OT1}{rsfs}{m}{n}{ <-7> rsfs5 <7-10> rsfs7 <10->
rsfs10}{} \DeclareMathAlphabet{\mycal}{OT1}{rsfs}{m}{n}

\newcounter{mnotecount}[section]

\newcommand{\N}{{\Bbb N}}

\newcommand{\rmnote}[1]{}

\newcommand{\Ric}{\operatorname{Ric}}

\def\mysavedown#1{\edef\mysubs{\mysubs#1}}
\def\mysaveup#1{\edef\mysups{\mysups#1}}
\def\mydown#1{{\mytensor}_{\vphantom{\mysubs}#1}}
\def\myup#1{{\mytensor}^{\vphantom{\mysups}#1}}
\def\tensor#1#2{
  #1
  \def\mytensor{\vphantom{#1}}
  \def\mysubs{\relax}
  \def\mysups{\relax}
  \let\down=\mysavedown
  \let\up=\mysaveup
  #2
  \let\down=\mydown
  \let\up=\myup
  #2
  }

\newcommand{\Riem}{\operatorname{Riem}}
\newcommand{\Hess}{\operatorname{Hess}}

\newcommand{\Tr}{\operatorname{Tr}}

\newcommand{\R}{\mathbb R}

\renewcommand{\div}{\operatorname{div}}

\DeclareMathOperator{\Vol}{Vol}

\renewcommand{\phi}{\varphi}
\renewcommand{\epsilon}{\varepsilon}

\def\crn#1#2{{\vcenter{\vbox{
        \hbox{\kern#2pt \vrule width.#2pt height#1pt
           }
          \hrule height.#2pt}}}}

\newcommand{\Ein}{\operatorname{Ein}}

\renewcommand{\hbar}{{\overline h}}

\newcommand{\pre}[2]{{{\vphantom{#2}}^{#1}}\kern-.2ex{#2}}

\theoremstyle{plain}
\newtheorem{theorem}{Théorème}[section]
\newtheorem{lemma}[theorem]{Lemme}
\newtheorem{proposition}[theorem]{Proposition}

\theoremstyle{definition}
\newtheorem{exemple}[theorem]{Exemple}

\newtheorem{remark}[theorem]{Remarque}

\numberwithin{equation}{section}

\begin{document}
\title[Inversion d'opérateurs de courbure]
{Inversion d'opérateurs de courbure au voisinage d'une métrique Ricci parallèle}

\author[E. Delay]{Erwann Delay}
\address{Erwann Delay,
Labo. de Math\'ematiques d'Avignon,
 Fac. des Sciences,
 F84916 Avignon, France}
\date{2 mai 2016}
\email{Erwann.Delay@univ-avignon.fr}
\urladdr{http://www.univ-avignon.fr/fr/recherche/annuaire-chercheurs/\newline$\mbox{ }$
\hspace{3cm} membrestruc/personnel/delay-erwann-1.html}

\begin{abstract}
Soit $(M,g)$  une variété riemannienne compacte sans bord, à courbure de Ricci
 parallèle. 
Nous montrons que certains opérateurs, affines en la courbure de Ricci, sont localement
inversibles, au voisinage de la métrique $g$.
\end{abstract}


\maketitle

\noindent {\bf Mots clefs} : Courbure de Ricci, variété produit, métriques d'Einstein,
2-tenseurs symétriques, EDP elliptique quasi-linéaire.
\\
\newline
{\bf 2010 MSC} : 53C21, 53A45,  58J05, 58J37, 35J62.
\\
\newline

\tableofcontents

\section{Introduction}\label{section:intro}
Sur  une variété Riemannienne $(M,g)$, considérons $\Ric(g)$ sa courbure de Ricci  et $R(g)$ sa courbure scalaire.
Parmi les (champs de) 2-tenseurs symétriques géométriques naturels que l'on peut construire,
les plus simples sont ceux qui seront "affines" en la courbure de Ricci, autrement dit, de la forme
$$
\Ein(g):=\Ric(g)+\kappa R(g)g+\Lambda g,
$$
o\`u $\kappa$ et $\Lambda$ sont des constantes.
Ainsi, si $\kappa=\Lambda=0$ on retrouve la courbure de Ricci, si $\kappa=-\frac12$  le tenseur d'Einstein (avec constante cosmologique $\Lambda$), enfin si $\kappa=-\frac1{2(n-1)}$ et $\Lambda=0$ le 
tenseur de Schouten.
Rappelons que ce tenseur est géométriquement  naturel dans le sens o\`u  pour tout difféomorphisme $\varphi$ assez régulier,
$$
\varphi^*\Ein(g)=\Ein(\varphi^*g).
$$
Nous nous posons ici le problème de l'inversion de l'opérateur $\Ein$.
On se donne donc $E$ un champ de tenseur symétrique sur $M$,  on cherche $g$ métrique riemannienne  telle
\bel{mainequation}
\Ein(g)=E.
\ee
On doit ainsi résoudre un système quasi-linéaire particulièrement complexe.
Le cas de la courbure de Ricci prescrite remonte
aux  années 80. 
DeTurck \cite{{Deturck:ricci}}, en 1981, a tout d'abord montr\'e un r\'esultat
d'existence locale au voisinage d'un point $p$ dans $\R^n$ sous
l'hypoth\`ese (intrins\`eque) que la matrice de  $R(p)$ est
inversible (il a depuis entrepris une longue étude systématique
pour le cadre local, comme le montrent ses  travaux en 1999
\cite{Deturckrank1}).

Puis il y a eu des  r\'esultats {\it globaux} :  DeTurck
\cite{Deturckdim2}, en 1982, a trait\'e le cas tr\`es particulier de la
dimension 2, pour les surfaces {\it compactes}. Il obtient une
condition
 n\'ecessaire et suffisante faisant intervenir la caract\'eristique
d'Euler-Poincar\'e. Hamilton \cite{Hamilton1984}, en 1984, a trait\'e le cas
de la sph\`ere unit\'e de $\R^{n+1}$ (avec $n>2$) en prouvant un
r\'esultat d'inversion locale au voisinage de la m\'etrique
standard.
Nous avions ensuite prouvé un résultat analogue sur l'espace hyperbolique réel \cite{Delay:etude},
et complexe \cite{DelayHerzlich}, au voisinage de la métrique canonique.
Parmis les résultats récents on peut aussi citer les travaux de A. Pulemotov 
comme \cite{Pulemotov:RicciBord}.


Ce type d'inversion locale  a été ensuite adapté à certaines variétes d'Einstein \cite{DeturckEinstein}, \cite{Delay:study}, \cite{Delanoe2003}.

Notons qu'il existe aussi des r\'esultats d'obstruction sur l'inversion de la courbure de
Ricci  \cite{Deturck-Koiso}, \cite{Baldes1986},
 \cite{Hamilton1984}, \cite{Delanoe1991}, \cite{Delay:etude}.

 Afin d'illustrer simplement le type de résultats obtenus ici,  nous en donnerons 
un corollaire  :

\begin{theorem}\label{maintheorem} 
Soit $(M,g)$ une variété riemannienne lisse  dont la courbure $\Ein(g)$ est non dégénérée et parallèle. On suppose que $\kappa=0$ et
que $-2\Lambda$ n'est pas dans le spectre du laplacien de Hodge
agissant sur les 1-formes, ni dans le spetcre du laplacien de Lichnerowicz.
Soient  $k\in\N\backslash\{0\}$ et $\alpha\in(0,1)$.  
Alors pour tout $e\in C^{k+2,\alpha}(M,\mathcal S_2)$ proche de zéro,  il existe un unique $h$ proche de zéro dans 
$C^{k+2,\alpha}(M,\mathcal S_2)$ telle que
$$
\Ein(g+h)=\Ein(g)+e.
$$
De plus l'application $e\mapsto h$ est lisse au voisinage de zéro entre les espaces de Banach correspondants. 
\end{theorem}

Ce théorème est un cas particulier du théorème  \ref{theoinvEin} o\`u $\kappa$ n'est pas forcément nulle et $-2\Lambda$  peut \^etre dans spectre du laplacien de Lichnerowicz,
à condition de modifier l'équation (\ref{mainequation}) par l'ajout d'une projection.

Nous étendons ainsi les résultats antérieurs  sur les variétés compactes d'Einstein, de courbure  positive, au cas  de métriques Ricci parallèles, et de tout types de courbures.
Noter aussi l'apparition d'une condition topologique via le laplacien de Hodge (annulation du premier nombre de Betti si $\kappa=\Lambda=0$).

La régularité de notre solution est optimale, il suffit de transporter l'équation par un difféomorphisme
peu régulier pour s'en convaincre.

Le fait que la métrique de départ soit Ricci parallèle équivaut au fait  qu'elle est localement le produit de métriques d'Einstein
(voir par exemple \cite{Wu:Holonomy}).

Un exemple modèle pour cet article est le produit de deux variétés compactes d'Einstein $(M_1,g_1)$ et $(M_2,g_2)$, ainsi 
$$M=M_1\times M_2\;,\;\;
g=g_1\oplus g_2.$$
Le cas particulier $\kappa=\Lambda=0$, donc de prescription de la courbure de Ricci,
doit servir de fil conducteur.
Dans ce cas le noyau du laplacien de Lichnerowicz $\Delta_L$ est de dimension au moins 2 puisqu'il contient $c_1\;g_1\oplus c_2\;g_2$, $c_1,c_2\in\R$.

Concernant les travaux précédents sur les variétés compactes, o\`u  la dimension du noyau de $\Delta_L$ est supposée \^etre égal à 1,
 on peut remplacer la résolution qui sera  donnée ici "à une constante additive près" par une résolution à "une constante multiplicative près", voir  la remarque \ref{rem:mult}.

Après une étude plus précise de la positivité de $\Delta_L$ et de son  noyau, nous donnons  des exemples de variétés à courbure sectionnelle  positive (au sens large) qui vérifient les hypothèses. 
Par souci pédagogique, nous commencons   par l'inversion de l'opérateur de Ricci 
en section \ref {sec:ppal}.

On traite aussi d'un résultat analogue pour l'opérateur de Ricci contravariant, plus adapté au cas de la courbure négative.

Enfin on étudie la prescription des autres opérateurs de courbure, dont le théorème \ref{maintheorem} est en fait un cas particulier, o\`u  nous donnons
un exemple en courbure nulle.

Finalement, cette dernière étude nous permet de prouver que l'image de certains opérateurs de type Riemann-Christoffel sont
des sous-variétés lisses dans $C^\infty$.

\medskip

{\small\sc Remerciements} : Ce travail est en partie financé par les ANR  SIMI-1-003-01 et ANR-10-BLAN 0105.

\section{Définitions, notations et conventions}\label{sec:def}

Nous noterons  $\nabla$ la connexion  de Levi-Civita de $g$, par $\Ric(g)$   sa courbure de Ricci et par
$\Riem(g)$ sa courbure de  Riemannian sectionelle. 

Soit ${\mathcal T}_p^q$ l'ensemble des tenseurs covariants de rang $p$.
Lorsque $p=2$ et $q=0$, on notera ${\mathcal S}_2$ le sous ensemble des tenseurs symétriques
qui se décompose et ${\mathcal G}\oplus {\mathring{\mathcal
S}_2}$ o\`u ${\mathcal G}$ est l'ensemble des  tenseurs  $g$-conformes et 
${\mathring{\mathcal S}_2}$ l'ensemble des tenseurs sans traces (relativement à $g$). On utilisera la convention de sommation  d'Einstein
 (les indices correspondants vont de $1$ à $n$), et nous utiliserons 
 $g_{ij}$ et son inverse $g^{ij}$ pour monter ou descendre les indices.

Le Laplacian (brut) est définit par
$$
\triangle=-tr\nabla^2=\nabla^*\nabla,
$$
o\`u $\nabla^*$ est l'adjoint formel $L^2$ de $\nabla$. Le Laplacian de
Lichnerowicz  agissant sur les (champs de) 2-tenseurs covariant symétriques est
\bel{laplichne}
\triangle_L=\triangle+2(\Ric-\Riem),
\ee
o\`u
$$(\Ric\; u)_{ij}=\frac{1}{2}[\Ric(g)_{ik}u^k_j+\Ric(g)_{jk}u^k_i],
$$
et
$$
(\Riem \; u)_{ij}=\Riem(g)_{ikjl}u^{kl}.
$$
Pour  $u$ un 2-tenseur covariant symétrique, on définit sa divergence par
 $$ (\mbox{div}u)_i=-\nabla^ju_{ji}.$$ Pour une 1-forme
$\omega$ on $M$, on définit sa divergence par :
$$
d^*\omega=-\nabla^i\omega_i,
$$
et la partie symétrique  de ses dérivées covariantes:
$$
({\mathcal
L}\omega)_{ij}=\frac{1}{2}(\nabla_i\omega_j+\nabla_j\omega_i),$$
(notons que ${\mathcal L}^*=\mbox{div}$).
Le laplacien de Hodge-de Rham agissant sur les 1-formes sera noté
$$
\Delta_H=dd^*+d^*d=\Delta+\Ric.
$$

On définit l'opérateur de Bianchi des 2-tenseurs symétriques dans les 1-formes :
$$
B_g(h)=\div_gh+\frac{1}{2}d(\Tr_gh).
$$

\section{Laplacien de Lichnerowicz et isomorphisme}
Nous commencerons cette section par une autre écriture du Laplacien de
Lichnerowicz. Elle  permet entre autre de voir simplement  que les tenseurs parallèles sont dans son noyau.
On considère l'opérateur de ${\mathcal S}_2$ dans ${\mathcal T}_3$
définit par
$$
(Du)_{kij}:=\frac{1}{\sqrt{2}}(\nabla_ku_{ij}-\nabla_ju_{ik}),
$$
cet opérateur étant , à une constante près, la différentelle extérieure de $u$
vue comme une 1-forme à valeur dans le cotangent
(voir  \cite{Besse} 1.12. p.24).
L'adjoint formel de  $D$  est
$$
(D^*T)_{ij}=\frac{1}{2\sqrt{2}}(-\nabla^kT_{kij}-\nabla^kT_{kji}
+\nabla^kT_{ijk}+\nabla^kT_{jik}).
$$
Ainsi on a 
$$
D^*Du_{ij}=-\nabla^k\nabla_ku_{ij}+\frac{1}{2}(\nabla^k\nabla_iu_{jk}+\nabla^k\nabla_ju_{ik}).
$$
Remarquons d'autre part que
$$
\mathcal L\mathcal L^*u_{ij}=-\frac{1}{2}(\nabla_i\nabla^ku_{jk}+\nabla_j\nabla^ku_{ik}).
$$
et que
$$
\nabla^k\nabla_ju_{ik}-\nabla_j\nabla^ku_{ik}=\Ric(g)_{qj}u^q_i-
   \Riem(g)_{qilj}u^{ql}
$$
On  obtient ainsi la formule de  Weitzenb\"ock  :
$$
\triangle_K:=D^*D+\mathcal L\mathcal L^*=\nabla^*\nabla+\Ric-\Riem.
$$
Par conséquent on a :
\bel{lichneparal}
\Delta_L=2(D^*D+\mathcal L\mathcal L^*)-\nabla^*\nabla
\ee
Si $(M,g)$ est une variété compacte, nous noterons $\Pi$ la projection
orthogonale $L^2$ sur $\ker\Delta_L$. Ainsi, si $h_1,...,h_k$ est une base $L^2$-orthonormée
de $\ker\Delta_L$, 
$$
\Pi(h)=\sum_{i=1}^k\langle h,h_i\rangle_{L^2}h_i.
$$
Nous pouvons énoncer la
\begin{proposition}\label{DeltaLiso}
Soient $k\in\N$ , $\alpha\in(0,1)$ et $c$ un réel non nul. 
L'opérateur 
$\Delta_L+c\Pi$ est un isomorphisme de $C^{k+2,\alpha}(M,\mathcal S_2)$ dans $C^{k,\alpha}(M,\mathcal S_2)$.
\end{proposition}
\begin{proof}
Le preuve est classique, nous en donnerons juste les grandes lignes.
Notons $\mathcal K$ le noyau de dimension finie de $\Delta_L$, ces éléments sont lisses par régularité elliptique.
On note $\mathcal K^\perp$, l'orthogonal $L^2$  de $\mathcal K$.
Alors   
$$\Delta_L: H^2\cap\mathcal K^\perp\longrightarrow \mathcal K^\perp
$$
est un isomorphisme (voir par exemple les théorèmes 31 et 27 pages 463-364 de \cite{Besse}) . 
Ensuite tout élément $h\in H^2$ se décompose en $$h=u^\perp+u\in (H^2\cap\mathcal K^\perp)\oplus\mathcal K.$$
L'application $$h\mapsto \Delta_L(u^\perp)+cu\in K^\perp\oplus\mathcal K$$ est clairement
un isomorphisme, or c'est $\Delta_L+c\Pi$. Il suffit ensuite d'utiliser
la régularité elliptique pour conclure à l'isomorphisme entre les espaces de H\"older.
\end{proof}
Étudions maintenant plus précisément le noyau du Laplacien de {Lichnerowicz}.
\begin{lemma}\label{Bgnul}
$\mbox{ }$

i) Si   $\Ric-\Riem$ est positif sur $L^2$ alors $\ker\Delta_L=\ker\nabla$.

ii) Si   La courbure de Ricci est parallèle et que le premier nombre de Betti est nul
alors $\ker\Delta_L\subset\ker\div.$ 

 Dans les deux cas  nous avons 
$$\ker\Delta_L\subset \ker\div\cap\ker (d\circ \Tr)\subset \ker B_g,$$ en particulier
$$
B_g\circ\Pi=0.
$$
\end{lemma}
\begin{proof}
Par l'écriture \eq{lichneparal}, il est clair qu'on a toujours $\ker\nabla\subset\ker\Delta_L$.
Pour l'autre inclusion, il suffit revenir à la définition de $\Delta_L$ \eq{laplichne}.
Enfin si la courbure de Ricci est parallèle,  par \cite{Lichnerowicz:prop} on a 
\bel{divDeltaL}
\Delta_H\circ\div=\div\circ\Delta_L.
\ee
Pour la dernière inclusion rappelons juste que $\Delta_L$ respecte
la décomposition $S_2=\mathcal G\oplus\mathring S_2$ avec 
\bel{TrDeltaL}
\Tr\circ\Delta_L=\Delta\circ\Tr.
\ee
\end{proof}

Afin de connaître aussi   l'éventuelle  positivité de $\Delta_L$, étudions celle de
 $\Ric-\Riem$. Nous donnons pour cela un  lemme algébrique.
\begin{lemma}\label{lem:fujitani}
Soit un point $x$ de $M$, on note $\Ric_{\min}$ la plus petite valeurs propre de  $\Ric(g)$ en $x$, 
$K_{\max}$ et $K_{\min}$  le max et le min de la courbure sectionnelle en $x$. Alors pour tout $h\in \mathring S_2$  en $x$, on a l'inégalité ponctuelle:
$$
\langle (\Ric-\Riem) h,h\rangle_{g_x}\geq\max\{2\Ric_{\min}-(n-2)K_{\max},nK_{\min}\}\|h\|^2_{g_x}.
$$
\end{lemma}
\begin{proof}
La preuve est inspirée du lemme de Fujitani (\cite{Besse} p.356) (voir 	aussi dans \cite{Deturck-Koiso} la  preuve du corollaire 3.4 p 356) mais adaptée au cas non forcément Einstein.
Choisissons $h\in \mathring S_2$ un tenseur propre de $\Ric-\Riem$. Prenons  une base orthonormée o\`u $h$ est diagonale, de valeurs propres 
$\lambda_1,...,\lambda_n$ avec $\lambda_1=\sup|\lambda_i|$ (et $\sum \lambda_i=0$).
On pose 
$$
a=\frac{\langle (\Ric-\Riem) h,h\rangle_{g_x}}{\|h\|^2_{g_x}}.
$$
On a 
$$
\begin{array}{lll}
a\lambda_1&=&[(\Ric-\Riem) h]_{11}=\sum_lR_{1l}h_{1l}-\sum_{i,k}R_{i1k1}h_{ik}\\
&=&R_{11}\lambda_1-\sum_{i\neq1}R_{i1i1}\lambda_i\\
&=&R_{11}\lambda_1-\sum_{i\neq1}K_{\max}\lambda_i+\sum_{i\neq1}(K_{\max}-R_{i1i1})\lambda_i\\
&=&R_{11}\lambda_1+K_{\max}\lambda_1+\sum_{i\neq1}(K_{\max}-R_{i1i1})\lambda_i\\
&\geq&R_{11}\lambda_1+K_{\max}\lambda_1-\sum_{i\neq1}(K_{\max}-R_{i1i1})\lambda_1\\
&=&2R_{11}\lambda_1-(n-2)K_{\max}\lambda_1\\
&\geq&(2\Ric_{\min}-(n-2)K_{\max})\lambda_1\;.\\
\end{array}
$$
Le même type de raisonnement donne 
$$
\begin{array}{lll}
a\lambda_1
&=&R_{11}\lambda_1-\sum_{i\neq1}K_{\min}\lambda_i-\sum_{i\neq1}(R_{i1i1}-K_{\min})\lambda_i\\
&\geq&R_{11}\lambda_1+K_{\min}\lambda_1-\sum_{i\neq1}(R_{i1i1}-K_{\min})\lambda_1\\
&=&nK_{\min}\lambda_1\;.\\
\end{array}
$$

\end{proof}

\section{Cas de la courbure de Ricci}\label{sec:ppal}
Il est maintenant bien connu que l'équation de Ricci n'est pas elliptique du à l'invariance
de la courbure par difféomorphisme. Nous allons modifier cette  équation en s'inspirant
de la m\'ethode de DeTurck. On  y ajoute donc un terme jauge de telle sorte
que le cette nouvelle equation devienne elliptique tout en faisant en sorte
que ses solutions soient solution de l'équation de Ricci.
Nous devons aussi ici prendre en compte le fait que le Laplacien de Lichnerowicz
peut avoir un noyau de dimension plus grande que 1, contrairement
aux travaux précédents.
Afin de construire notre nouvelle équation, rappelons quelques différentielles d'opérateurs.\\

Nous avons déjà (voir \cite{Besse} par exemple )
$$
D\Ric(g)h=\frac12\Delta_Lh-\mathcal L_gB_g(h).
$$
Le linéarisé en la première variable de l'opérateur de Bianchi est (voir par exemple \cite{Delay:study})
$$
[D B_{(.)}(R)](g)h=-{R}B_g(h)+T(g,R)h,
$$
o\`u ${R}$ est identifié ici à l'endomorphisme de $T^*M$ correspondant et
$$
[T(g,R)h]_j=T(g,R)^{kl}_j h_{kl}=\frac{1}{2}(\nabla^k R^l_j+\nabla^l
  R^k_j-\nabla_jR^{kl})
h_{kl}
$$
En particulier, si $g$ est Ricci paralllèle on a 
$$
[D B_{(.)}(\Ric(g))](g)h=-{\Ric_g}B_g(h).
$$
o\`u est $\Ric_g$ l'endomorphisme de $T^*M$ donné par $(\omega_i)\mapsto(\Ric(g)^k_i\omega_k)$.

Rappelons enfin que pour toute métrique $g$, $B_{g}(\Ric(g))=0$
par l'identité de Bianchi.

L'équation que nous choisissons de résoudre sera

\bel{equaelliptique}
F(h,r):=\Ric(g+h)-R(h,r)-{\mathcal L}_g\{\Ric_g^{-1}B_{g+h}[R(h,r)]\}=0,
\ee
o\`u 
$$
R(h,r)=Ric(g)+r-\frac12\Pi(h),
$$
 et $\Pi$ est la projection orthogonale $L^2$ sur ker$\Delta_L^g$.

Commen\c{c}ons par vérifier que les solutions de la nouvelle equation sont solutions de
l'équation qui nous intéresse.
\begin{proposition}\label{sol}
Sous les conditions du théorème \ref{maintheorem}, avec $\kappa=\Lambda=0$,  si
  $h\in C^{k+2,\alpha}(M,\mathcal S_2)$ est assez petit, et que la métrique
$g+h$ est solution de \eq{equaelliptique}, alors c'est une solution de 
$$
\Ric(g+h)=R(h,r).
$$
\end{proposition}
\begin{proof}
On applique $B_{g+h}$ à l'équation \eq{equaelliptique}. Remarquons que $B_{g+h}[\Ric(g+h)]=0$ par l'identité de Bianchi. Ainsi, si on pose 
$$\omega:=\Ric_g^{-1}B_{g+h}(R(h,r)),$$ on obtient
$$
P_{g+h}\omega:=B_{g+h}[{\mathcal L}_g(\omega)]+\Ric_g\omega=0.
$$
L'opérateur $P_g$ se lit en coordonnées locales:
$$
(P_g\omega)_j=-\nabla^{i}\left[\frac{1}{2}(\nabla_i\omega_j+\nabla_j\omega_i)\right]
+\frac{1}{2}\nabla_j\nabla^i\omega_i+\Ric(g)_{j}^k\omega_k.
$$
Commutons les dérivées et  multiplions par  $2$, on obtient 
$$2P_g=\Delta_g\omega+\Ric_g\omega=\Delta_H \omega.$$
Comme le premier nombre Betti est nul ($b_1=0$) l'opérateur $P_g$ a un noyau $L^2$ trivial (et c'est un isomorphisme de
$C^{k+1,\alpha}(M,\mathcal T_1)$ dans $C^{k-1,\alpha}(M,\mathcal T_1)$). Maintenant si 
 $h$ est petit dans $C^{k+2,\alpha}(M,\mathcal S_2)$, l'opérateur $P_{g+h}$ reste injectif. On peut donc conclure que $\omega=0$.
\end{proof}

\begin{remark}
Le fait que $B_{g+h}[R(h,r)]$ s'annule prouve que l'application identité de $(M,g+h)$ dans
$(M,R(h,r))$ est harmonique (voir \cite{GL} par exemple).
\end{remark}

Nous allons maintenant construire les solutions de \eq{equaelliptique} par un argument de fonctions implicites 
dans des espaces de Banach.
\begin{proposition}
Sous les conditions du théorème \ref{maintheorem}, avec $\kappa=\Lambda=0$.
Pour tout $r\in C^{k+2,\alpha}(M,\mathcal S_2)$ petit, il existe un unique 
$h$ proche de zéro dans  $ C^{k+2,\alpha}(M,\mathcal S_2)$ solution de \eq{equaelliptique}.
\end{proposition}

\begin{proof}
On considère $F$ comme application définie au voisinage de zéro
dans $C^{k+2,\alpha}(M,\mathcal S_2)\times C^{k+2,\alpha}(M,\mathcal S_2) $ à valeur dans 
$C^{k,\alpha}(M,\mathcal S_2)$.  On a déjà $F(0,0)=0$. Compte tenue des différentielles des opérateurs données  en début de section et du lemme \ref{Bgnul}, la différentielle relativement à $h$ en $0$ est
$$
D_{h}F(0,0)=\frac12\Delta_L+\frac12\Pi.
$$
Par la propositon \ref{DeltaLiso}, cet opérateur est un isomorphisme de $C^{k+2,\alpha}(M,\mathcal S_2)$ dans $C^{k,\alpha}(M,\mathcal S_2)$.
Le théorème des fonctions implicites permet de conclure.

\end{proof}

\begin{remark}\label{rem:mult}
$R(h,r)$ peut etre remplacée par un 2-tenseur tel que $R(0,r)=\Ric(g)+r$ et $D_hR(0,0)=c\Pi$, $c\neq 0$. Par exemple dans \cite{Delanoe2003} o\`u le noyau de $\Delta_L$
est réduit aux multiples de $g$,  on a 
$$
\Pi(h)=\frac1{n\Vol_g(M)}(\int_M\Tr_ghd\mu_g) g=:\frac1n<\Tr_gh>g,
$$ 
et comme $\Ric(g)=g$, on peut prendre $$R(h,r)=e^{\frac1n<\Tr_gh>}(\Ric(g)+r).$$

\end{remark}

\begin{exemple}
Rappelons tout d'abord qu'une métrique Ricci parallèle est forcément
localement le produit de variétés d'Einstein (voir par exemple \cite{Wu:Holonomy}).
En particulier si $g$ est un produit de métriques d'Einstein à courbures scalaires strictement positives,  alors  comme $\Ric(g)>0$,  le premier 
nombre Betti est nul et le théorème \ref{maintheorem} s'applique. Si de plus les courbures  sectionnelles sont  positives (ou nulles), alors par  le lemme \ref{lem:fujitani}, $ \Ric-\Riem\geq 0$ et
le noyau de $\Delta_L$ est réduit aux 2-tenseurs symétriques parallèles. 
\end{exemple}


\section{Opérateur de Ricci contravariant}
On s'intéresse ici à l'inversion de l'opérateur de Ricci
contravariant :$$g\mapsto \overline\Ric(g)$$
 dont les composantes  en coordonnées locales sont $\overline\Ric(g)^{ij}=g^{ik}g^{jl}\Ric(g)_{kl}$. Nous utiliserons la notation évidente 
$$\overline\Ric(g)=g^{-1}\Ric(g)g^{-1}.$$
Nous adaptons les étapes de la section \ref{sec:ppal}. Tout d'abord on a 
$$
D\overline\Ric(g)h=g^{-1}[\frac12\Delta_Lh-\mathcal L_g B_g(h)-2\Ric h]g^{-1}.
$$
Posons $\overline B_g(\overline R)=B_{g}(g\overline Rg)$, ainsi si
 $\nabla\overline  R=0$, on obtient
$$
D\overline B_{(.)}(\overline R)=-g\overline RB_g(h)+B_g(h\overline Rg+g\overline R h).
$$
Si de plus $\overline R=\lambda g^{-1}$, on trouve
$$
D\overline B_{(.)}(\overline R)h=\lambda B_g(h)=g\overline RB_g(h).
$$
L'équation avec jauge que nous choisissons de résoudre ici sera 
\bel{equaelliptiquecontra}
\overline F(h,\overline r):=g[\overline\Ric(g+h)-\overline R(h,\overline r)]g+{\mathcal L}_g\{\Ric_g^{-1}\overline B_{g+h}[\overline R(h,\overline r)]\}=0,
\ee
o\`u 
$$
\overline R(h,\overline r)=\overline {\Ric}(g)+\overline r-\frac12g^{-1}\overline \Pi(h)g^{-1},
$$
 et $\overline \Pi$ est la projection orthogonale $L^2$ sur ker$(\Delta_L^g-4\Ric)$.

\begin{theorem}\label{theoremRicContra}
Soient  $k\in\N\backslash\{0\}$ et $\alpha\in(0,1)$. Soit $g$ une métrique d'Einstein à courbure scalaire non nulle.
On suppose que
le noyau de  l'opérateur $\Delta_H-4\Ric_g$ est trivial, ainsi que celui de $\Delta-4\frac{R(g)}n$ agissant sur les fonctions .  Alors pour tout $\overline r\in C^{k+2,\alpha}(M,\mathcal S^2)$ proche de zéro,  il existe un unique $h$ proche de zéro dans 
$C^{k+2,\alpha}(M,\mathcal S_2)$ telle que
$$
\overline \Ric(g+h)=\overline\Ric(g)+\overline r-\frac12g^{-1}\overline\Pi(h)g^{-1}.
$$
De plus l'application $\overline r\mapsto h$ est lisse au voisinage de zéro entre les Banach correspondants. 
\end{theorem}
\begin{proof}
On a encore $\overline F(0,0)=0$ et comme $g$ est d'Einstein, par les hypothèses sur les noyaux et en utilisant les formules
(\ref{divDeltaL}) et (\ref{TrDeltaL}) on trouve que les éléments du noyau de $\Delta_L-4\Ric$ sont à divergence nulle
et trace nulle (dit aussi TT-tenseurs), en particulier $B_g\circ\overline\Pi=0$. On trouve ainsi
$$
D_h\overline F(0,0)=\frac12\Delta_L-2\Ric+\frac12\overline\Pi.
$$
L'analogue de la proposition (\ref{DeltaLiso}) prouve que cet opérateur est un isomorphisme de $C^{k+2,\alpha}(M,\mathcal S_2)$ dans $C^{k,\alpha}(M,\mathcal S_2)$.
Par le théorème des fonctions implicites, pour tout $\overline r\in C^{k+2,\alpha}(M,\mathcal S^2)$ petit, il existe 
$h$ proche de zéro dans 
$C^{k+2,\alpha}(M,\mathcal S_2)$ tel que
$$
\overline F(h,\overline r)=0.
$$
On  
applique ensuite $\overline B_{g+h}$ à l'équation \eq{equaelliptiquecontra}, on obtient
$$
\overline P_{g+h}\omega:=B_{g+h}\mathcal L_g\omega-\Ric_g \omega=0
$$
o\`u 
$$
\omega=\Ric_g^{-1}\overline B_{g+h}(\overline R(h,\overline r)).
$$
Or par hypothèse,
$$
\overline P_g=\frac12(\Delta-\Ric_g)-\Ric_g=\frac12(\Delta_H-4\Ric_g)
$$
est injectif, ainsi si $h$ est assez petit $\overline P_{g+h}$ le reste, donc $\omega=0$. 
\end{proof}
\begin{exemple}
Une variété compacte   d'Einstein à courbure scalaire strictement
négative, par exemple normalisée par
$$
\Ric(g)=-g.
$$
satisfait les hypothèses. 
Si l'on veut en savoir un peu plus sur la positivité et la projection,
remarquons tout d'abord que 
$$\Delta_L-4\Ric=\Delta-2(\Ric+\Riem).$$
Ainsi par le lemme de Fujitani (\cite{Besse} p 356), si en plus la courbure sectionnelle est négative (ou nulle), on a $\Riem\leq1$ et comme  $\Ric=-1$, on en déduit que $\Delta_L-4\Ric\geq\Delta\geq0$. 
Dans ce cas on a en particulier $\ker(\Delta_L-4\Ric)\subset\ker\nabla$ mais comme $\ker\nabla\subset\ker\Delta_L$ (rappelons la formule (\ref{lichneparal})) on trouve
 $\ker(\Delta_L-4\Ric)=\{0\}$ et l'inversion a lieu sans projection sur le noyau.
\end{exemple}

\section{Autres opérateurs de courbure}
Nous montrons ici que la méthode de la section \ref{sec:ppal} peut aussi \^etre adaptée à d'autres opérateurs, affines en la courbure de Ricci. Pour $\kappa$ et $\Lambda$ deux constantes réelles, on définit le  tenseur
$$
\Ein(g):=\Ric(g)+\kappa R(g)g+\Lambda g.
$$
Ainsi par exemple lorsque  $\kappa=-\frac12$ on retrouve le tenseur d'Einstein (avec constante cosmologique $\Lambda$), et si $\kappa=-\frac1{2(n-1)}$ et $\Lambda=0$ le 
tenseur de Schouten.
On étudie l'inversion de l'opérateur $\Ein$. On se donne donc $E$ un 2-tenseur symétrique et l'on cherche $g$ telle que
\bel{mainequationE}
\Ein(g)=E.
\ee
Comme nous avons
$$
\Tr_g\Ein(g)=(1+n\kappa)R(g)+n\Lambda,
$$
l'équation (\ref{mainequationE}) est équivalente à
$$
\Ric(g)=E-\frac{\kappa\Tr_g  E+\Lambda}{1+n\kappa}g
$$
Pour $E$ quelconque, on définit
$$
\mathcal B_g(E)=\div_gE+\frac{2\kappa+1}{2(1+\kappa n)}d\Tr_gE=B_g(E)-\frac{(n-2)\kappa}{2(1+\kappa n)}d\Tr_gE,
$$
de telle sorte que l'identité de Bianchi se traduise ici par
$$
\mathcal B_g(Ein(g))=0.
$$
Connaissant déjà la différentielle de $B_g(E)$ relativement à la métrique (voir \cite{Delay:study} par exemple), 
on trouve que la différentielle de cet opérateur relativement à la métrique est 
$$
D[\mathcal B_{(.)}(E)](g)h=-EB_g(h)+\frac{(n-2)\kappa}{2(1+\kappa n)}d\langle E,h\rangle
+T(E,h),$$
où $E$ est identifié ici à l'endomorphisme de $T^*M$ correspondant et
$$
T(E,h)_j=\frac12(\nabla_kE_{jl}+\nabla_lE_{kj}-\nabla_jE_{kl})h^{kl}.
$$
Par analogie avec la section \ref{sec:ppal}, on définit
$$
\mathcal F(h,e):=\Ric(g+h)-E+\frac{\kappa\Tr_{g+h}E+\Lambda}{1+\kappa n}{(g+h)}-\mathcal L_{g}\Ein_g^{-1}\mathcal B_{g+h}(E),
$$
où $\Ein_g$ est l'endomorphisme de $T^*M$ associé à $\Ein(g)$, $$E=\Ein(g)+e-\frac12\widetilde\Pi(h),$$ et $\widetilde \Pi$ une projection $L^2$ sur un espace de dimension fini
à préciser ultérieurement.
On a déjà
$$
\mathcal F(0,0)=0.
$$
Si la courbure de Ricci est parallèle, on a $\nabla\Ein(g)=0$  et  si l'on suppose $\mathcal B_g\circ\widetilde\Pi=0$,
on obtient
$$
D_h\mathcal F(0,0)h=$$
$$
\frac12\Delta_Lh+\frac12\widetilde\Pi(h)+\frac{1}{1+\kappa n}
\left(\kappa\Tr_g\Ein(g) \;h+\Lambda h-\kappa\langle \Ein(g),h\rangle g-\frac12\kappa\Tr_g\widetilde\Pi(h)\;g\right)
$$
$$
-\frac{(n-2)\kappa}{2(1+\kappa n)}\mathcal L_{g}\Ein_g^{-1}d\langle \Ein(g),h\rangle
$$

Ainsi si $g$ est d'Einstein $\Ein(g)=\tau g$ ou si $\kappa=0$, on a
$$
D_h\mathcal F(0,0)h=$$
$$
\frac12\Delta_Lh+\frac12\widetilde\Pi(h)+\frac{1}{1+\kappa n}
\left(n\kappa\tau \;h+\Lambda h-\kappa\tau\Tr_gh\;g-\frac12\kappa\Tr_g\widetilde\Pi(h)\;g\right)
$$
$$
-\frac{(n-2)\kappa}{2(1+\kappa n)}\nabla\nabla\Tr_gh.
$$
Cette différentielle nous incite à définir l'opérateur
$$
\begin{array}{lll}
\mathcal P h&:=&\Delta_Lh+\frac{2(n\kappa\tau+\Lambda)}{1+kn}h
+\frac{\kappa}{n(1+\kappa n)}\Big({(n-2)}\Delta \Tr_gh-2 n\tau\Tr_gh\Big)\;g\\
&=&(\Delta_L+{2\kappa R(g)+2\Lambda})h
+\frac{\kappa}{n(1+\kappa n)}\Big({(n-2)}\Delta \Tr_gh-2 n\tau\Tr_gh\Big)\;g.\\
\end{array}
$$
 Ce dernier respecte le scindage $\mathcal S_2=\mathcal G\oplus\mathring {\mathcal S}_2$. En particulier si
$u$ est une fonction sur $M$ et $\mathring h$ un champ de 2-tenseurs symétrique sans trace, on a
$$
\mathcal P(ug+\mathring h)=\frac{1}{1+\kappa n}p(u)g+\mathring P(\mathring h),
$$
o\`u
$$
p(u)=(1+2(n-1)\kappa)\Delta u+2\Lambda u,
$$
et 
$$
\mathring P(\mathring h)=\left[\Delta_L+{2\kappa R(g)+2\Lambda}\right]\mathring h.
$$

Pour $u$ une fonction sur $M$ et $\mathring h$ un champ de 2-tenseurs symétrique sans trace, on définit
$$\widetilde\Pi(ug+\mathring h):=\pi(u)g+\mathring\Pi(\mathring h),$$
o\`u $\pi$ est la projection $L^2$ sur noyau de $p$, et 
$\mathring\Pi$  la projection $L^2$ sur noyau de $\mathring P$.
Ainsi si  $h=ug$ on trouve 
 $$D_h\mathcal F(0,0)(ug)=
\frac1{2}[p(u) +\pi(u)]g-\frac{(n-2)n\kappa}{2(1+\kappa n)}\mathring\Hess \;u,
$$
o\`u $\mathring\Hess \;u$ est la partie sans trace de la hessienne de $u$.
Si $h=\mathring h$ est sans trace, on trouve
 $$
D_h\mathcal F(0,0)(\mathring h)
=\frac12\left(\mathring P+ \mathring\Pi\right)\mathring h.
$$
\begin{theorem}\label{theoinvEin}
Soient  $k\in\N\backslash\{0\}$, $\alpha\in(0,1)$, $\kappa\neq -\frac1n,-\frac1{2(n-1)}$ et  $\Lambda\in\R$. Soit $g$ une métrique Ricci parallèle si $\kappa=0$ et d'Einstein sinon, telle que $\Ein(g)$
est non dégénéré. On suppose que le noyau de $p$ est  trivial ou réduit aux constantes, 
et  que $\ker (\Delta_H+2\kappa R(g)+2\Lambda)=\{0\}$.  Alors pour tout $ e\in C^{k+2,\alpha}(M,\mathcal S_2)$ petit,  il existe un unique $h$ proche de zéro dans 
$C^{k+2,\alpha}(M,\mathcal S_2)$ telle que
$$
 \Ein(g+h)=\Ein(g)+e-\frac12\widetilde\Pi(h),
$$
De plus l'application $ e\mapsto h$ est lisse au voisinage de zéro entre les Banach correspondants. 
\end{theorem}

\begin{proof}
Les hypothèses sur les noyaux  garantissent que les éléments du noyau de $\mathcal P$
sont de trace constante et, en utilisant (\ref{divDeltaL}), à divergence nulle. On a bien ainsi $$\mathcal B_g\circ\widetilde\Pi=0.$$
Les analogues évident de la proposition \ref{DeltaLiso} prouvent que $p+\pi$ et $\mathring P+\mathring \Pi$
sont des isomorphismes de $C^{k+2,\alpha}$ dans $C^{k,\alpha}$ et donc que $D_h\mathcal F(0,0)$ aussi.
Les calculs qui précèdent et  le théorème des 
fonctions implicites impliquent alors que pour $ e\in C^{k+2,\alpha}(M,\mathcal S_2)$ petit,  il existe $h$ proche de zéro dans 
$C^{k+2,\alpha}(M,\mathcal S_2)$ tel que
$$
\mathcal F(h,e)=0.
$$
On applique maintenant $B_{g+h}$ à cette équation ainsi
$$
B_{g+h}\mathcal F(h,e)=-\mathcal B_{g+h}(E)-B_{g+h}\mathcal L_{g}\Ein_g^{-1}\mathcal B_{g+h}(E)=0.
$$
On pose $\omega=\Ein_g^{-1}\mathcal B_{g+h}(E)$ alors
$$
P_{g+h}\omega :=B_{g+h}\mathcal L_{g}\omega+\Ein_g\omega=0.
$$
Mais par hypothèse
$$
P_g=\frac12(\Delta-\Ric_g)+\Ein_g=\frac12(\Delta +\Ric_g+2\kappa R(g)+2\Lambda )=\frac12(\Delta_H +2\kappa R(g)+2\Lambda )
$$
est injectif, ainsi si $h$ est petit $P_{g+h}$ l'est encore donc $\omega=0$. 

\end{proof}

\begin{exemple}
Notons que quelque soit la courbure, si $\kappa>-1/2(n-1)$, quitte à prendre $\Lambda$ assez grand, tous les opérateurs seront strictement positifs
et l'inversion a lieu sans la projection.
Afin de donner un  exemple en courbure nulle, remarquons que pour  le tore plat, il suffit
de prendre $\Lambda>0$. Ce dernier exemple nous a poussé à étudier dans un autre article \cite{Delay:ricciAE}, une version asymptotiquement euclidienne
du théorème \ref{theoinvEin}.
\end{exemple}

\section{Image d'opérateurs de courbures de type Riemann-Christoffel}
Nous voudrions, tout comme dans \cite{Delay:etude} montrer que l'image de certain opérateurs de 
type Riemann-Christoffel, sont des sous variétés dans $C^\infty$, au voisinage de la  métrique $g$.
Nous cherchons donc tout d'abord un tenseur  $\mathcal Ein$ qui soit  4 fois covariant, ayant les m\^emes propriétés 
algébriques que le tenseur de Riemann et affine en la courbure, on 
pose donc 
$$
\mathcal Ein(g)=\Riem(g)+g {~\wedge \!\!\!\!\!\bigcirc ~} (a\Ric(g)+bR(g)g+cg), 
$$
o\`u ${~\wedge \!\!\!\!\!\bigcirc ~}$ est le produit de Kulkarni-Nomizu (\cite{Besse} p. 47).
Comme nous voulons que $\Tr_g\mathcal Ein(g)$ soit proportionnelle à $\Ein(g)$, cela nous impose 
 $$c=\frac{1+(n-2)a}{2(n-1)}\Lambda,\;\;b=\frac{\kappa[1+a(n-2)]-a}{2(n-1)}.$$
On a alors
$$
\Tr_g\mathcal Ein(g)=[a(n-2)+1]\Ein(g).
$$
Nous définirons la version de type Riemann-Christoffel de $\mathcal Ein(g)$ par
$$
[g^{-1}\mathcal Ein(g)]^i_{klm}:=g^{ij}\mathcal Ein(g)_{jklm}.
$$
Consid\'erons ${\mathcal R}^1_3$, le sous-espace de ${\mathcal T}^1_3$ des
tenseurs v\'erifiants
$$
\tau^i_{ilm}=0,\;\tau^i_{klm}=-\tau^i_{kml},\;
\tau^i_{klm}+\tau^i_{mkl}+\tau^i_{lmk}=0.
$$
On définit l'espace de Fréchet 
$$C^{\infty}=\cap_{k\in\N}C^{k,\alpha},$$
munit de la famille de semi-normes  $\{\|.\|_{k,\alpha}\}_{k\in\N}$.
On procède alors de façons similaire à \cite{Delay:etude} pour prouver que 
\begin{theorem}
Sous les conditions du théorème \ref{theoinvEin}, on suppose de plus que le noyau de $\mathcal P$ est trivial,
autrement dit  $\widetilde \Pi=0$. Alors l'image de l'application
$$
\begin{array}{lll}
C^{\infty}(\R^n,\mathcal S_2)&\longrightarrow&C^{\infty}(\R^n,\mathcal R_3^1)\\
h&\mapsto &(g+h)^{-1}\mathcal Ein(g+h)-(g)^{-1}\mathcal Ein(g)\\
\end{array}
$$
est une sous-variété lisse au voisinage de zéro.
\end{theorem}

\begin{remark}
Dans la définition de $\mathcal Ein$, le choix de $a\neq -1/2(n-1)$ est  encore libre.
Si nous  voulions retrouver la courbure de Riemann lorsque
$\kappa=\Lambda=0$ et, lorsque  $\kappa=-1/2$, un tenseur à divergence nulle (donc $a=-1$ et $b=1/4$ via l'identité de Bianchi 2). 
On pourrait choisir  par exemple  
$$
a=2\kappa\;,\;\;\;\; b=\frac{\kappa[2\kappa(n-2)-1]}{2(n-1)}\;,\;\;\;\;c=\frac{\Lambda[2\kappa(n-2)+1]}{2(n-1)}.
$$
Il n'est pas clair que ce choix soit plus naturel qu'un autre.
Peut être qu'une identité de type Bianchi 2 qui en découlerai serait aussi plus légitime mais
nous n'avons pas pu trancher à ce stade.

\end{remark}

%

\def\polhk#1{\setbox0=\hbox{#1}{\ooalign{\hidewidth
  \lower1.5ex\hbox{`}\hidewidth\crcr\unhbox0}}}
  \def\polhk#1{\setbox0=\hbox{#1}{\ooalign{\hidewidth
  \lower1.5ex\hbox{`}\hidewidth\crcr\unhbox0}}} \def\cprime{$'$}
  \def\cprime{$'$} \def\cprime{$'$} \def\cprime{$'$}
\providecommand{\bysame}{\leavevmode\hbox to3em{\hrulefill}\thinspace}
\providecommand{\MR}{\relax\ifhmode\unskip\space\fi MR }
\providecommand{\MRhref}[2]{%
  \href{http://www.ams.org/mathscinet-getitem?mr=#1}{#2}
}
\providecommand{\href}[2]{#2}

\end{document}